\documentclass[12pt,a4paper]{amsart}
\usepackage{a4,a4wide}
\usepackage{amsfonts,amsmath,amssymb,amsthm}

\usepackage{color}
\usepackage{ifpdf}
\ifpdf
    \usepackage[pdftex]{graphicx}
    \DeclareGraphicsExtensions{.png,.pdf,.jpg}
\else
   \usepackage[dvips]{graphicx}
   \DeclareGraphicsExtensions{.eps}
\fi
\graphicspath{{.}{figuresPL/}}
\usepackage[latin1]{inputenc}
\numberwithin{equation}{section}

\newtheorem{theorem}{Theorem}[section]
\newtheorem{lemma}[theorem]{Lemma}
\newtheorem{proposition}[theorem]{Proposition}
\newtheorem{corollary}[theorem]{Corollary}
\newtheorem{definition}[theorem]{Definition}
\newtheorem{remark}[theorem]{Remark}

\renewcommand\tilde{\widetilde}

\newcommand\calI{\mathcal{I}}

\def\R{\mathbb{R}}

\def\pa{\partial}

\newcommand\calA{\mathcal{A}}

\newcommand\m{\mu}

\def\E{\mathcal{E}}

\def\N{\mathbb{N}}

\def\H{\mathcal{H}}

\def\LM#1{\hbox{\vrule width.2pt \vbox to#1pt{\vfill \hrule width#1pt
height.2pt}}}
\def\LL{{\mathchoice {\>\LM7\>}{\>\LM7\>}{\,\LM5\,}{\,\LM{3.35}\,}}}
\def\restr{{\LL}}
\renewcommand{\phi}{\varphi}

\def\1{\mathbf{1}}

\def\XXint#1#2#3{{\setbox0=\hbox{$#1{#2#3}{\int}$ }
\vcenter{\hbox{$#2#3$ }}\kern-.57\wd0}}

\def\eps{\varepsilon}
\renewcommand{\subset}{\subseteq}

\def\lt{\left}
\def\rt{\right}
\def\les{\lesssim}
\def\ges{\gtrsim}

\def\weaklim{\rightharpoonup}

\def\l{T}

\def\BXi{\overline{X}_i}
\def\BXik{\BXi^k}

\def\Bl{T_*}
\newcommand\dotX{\dot{X}}

\def\beps{\overline{\eps}}
\def\bphi{\overline{\phi}}

\begin{document}
 \title{Self-similar minimizers of a branched transport functional}
\author{ Michael Goldman}
\address{LJLL, Universit\'e Paris Diderot,  CNRS, UMR 7598}
 \email{goldman@math.univ-paris-diderot.fr}
\date{}
\begin{abstract}\noindent
We solve here completely an irrigation problem from a Dirac mass to the Lebesgue measure for a two dimensional analog of a functional previously derived 
in the study of branched patterns in type-I superconductors. The minimizer we obtain is a self-similar tree.
\end{abstract}
\maketitle

\section{Introduction}
For a measure $\mu$ on $\R\times(a,b)$ such that $\mu=\mu_t\otimes dt$ with   $\mu_t=\sum_i \phi_i\delta_{X_i}$ for a.e. $t\in(a,b)$ for some (pairwise distinct) $X_i\in \R$, we  consider the branched transportation type functional (see Section \ref{secnot} for a more precise definition)
\begin{equation}\label{functional}
\E(\mu):= \int_{a}^b \sharp\{\phi_i\neq 0\} +\sum_i \phi_i |\dotX_i|^2 dt, 
\end{equation}
where   $\dotX_i$ denotes the time derivative of $X_i(t)$. This is indeed a branched transportation problem since by the Benamou-Brenier formula \cite{AGS, Sant,villani}, the second term is exactly the length of the curve $t\to \mu_t$ measured in the Wasserstein metric, while the first term forces concentration and thus branched structures.\\ 
Our main focus is the irrigation problem of the Lebesgue measure from a Dirac mass. By this we mean that we want to minimize the cost \eqref{functional} under the condition that the starting measure $\mu_a$ is a Dirac mass and that $\mu_t$ converges (weakly) to the Lebesgue measure as $t$ goes to $b$. This implies that the measure $\mu_t$ must infinitely refine as $t\to b$. 
We will also be interested in the case when both the initial and final measures are the Lebesgue measure. \\
More generally,   for two given  measures $\mu_{\pm}$ of equal mass, we study  the following Dirichlet problem 
\begin{equation}\label{minmainintro}
 \min_{\mu}\lt\{ \E(\mu) \ : \ \mu_a=\mu_- \ , \ \mu_b=\mu_+ \rt\},
\end{equation}
where the boundary condition is understood in the sense that $\mu_t\weaklim \mu_-$ as $t\to a$ and $\mu_t\weaklim \mu_+$ as $t\to b$.\\
Our main result is a full characterization of the minimizers of \eqref{minmainintro} in the case $\mu_-$ is a Dirac mass, $\mu_+$ is the Lebesgue measure restricted to an interval of length $\mu_-(\R)$ and $b-a$ is large enough. In order to fix notation, since the 
problem is invariant by translations, we may assume that $a=0$, $b=T$, $\mu_-= \phi \delta_{X}$ and $\mu_+=dx\restr[-\phi/2,\phi/2]$ for some $T, \phi>0$ and $X\in\R$. As will be apparent below, up to rescalings and shears, we may further normalize to $X=0$ and $\phi=1$, so that 
\[\mu_0=\delta_0 \qquad \textrm{and} \qquad \mu_T= dx\restr[-1/2,1/2].\]

In this case, as will become clearer in the proof, the threshold value $T=1/4$ naturally  appears. In order to state our main theorem, let us define for $t\in[0,1/4]$, the dyadically branching measure $\mu^*_t$ (see Figure \ref{figmin}). For $k\ge 0$, let $t_k:=\frac{1}{4}\lt(1-\left(\frac{1}{2}\rt)^{3k/2}\rt)$ be the branching times. 
We define recursively $\mu^*_t$ in the intervals $[t_{k-1},t_k]$.
Let $X_1^0=0$ and  $\mu^*_0=\delta_0$. Assume that $\mu^*_t$ is defined in $[0,t_{k-1}]$ and that $\mu^*_{t_{k-1}}=2^{-(k-1)}\sum_{i=1}^{2^{k-1}} \delta_{X_i^{k-1}} $.
For $t\in[t_{k-1},t_{k}]$ and $1\le i\le 2^k$, we now define $X_i^k(t)$. For this, let us divide $[-1/2,1/2]$ in $2^k$ intervals of equal size and let $\BXik$ be the barycenter of the $i-$th such interval i.e. 
 $\BXik:=\frac{-1}{2}+\frac{i-1}{2^{k}}+\frac{1}{2^{k+1}}$. We then let 
 \[X_i^k(t):=\frac{t-t_{k-1}}{\frac{1}{4}-t_{k-1}}\lt(\BXik-{X}_{ \lceil i/2\rceil}^{k-1}\rt)+{X}_{ \lceil i/2\rceil}^{k-1},\]
 and $\mu^*_t:= 2^{-k}\sum_{i=1}^{2^k}\delta_{X_i^k(t)}$. Notice that with this definition, for every $t\in[0,1/4)$, $k\in \N$ and $1\le i\le 2^k$, the mass at $X_i^k(t)$ is irrigating the interval
 $(\BXik-2^{-(k+1)},\BXik+2^{-(k+1)})$ and $X_i^k$ is moving at constant speed towards $\BXik$ (and would reach it at time $T=1/4$ if there were no further branching points). Our main theorem is the following
 \begin{figure}
  \begin{center}
  \includegraphics[height=6.5cm]{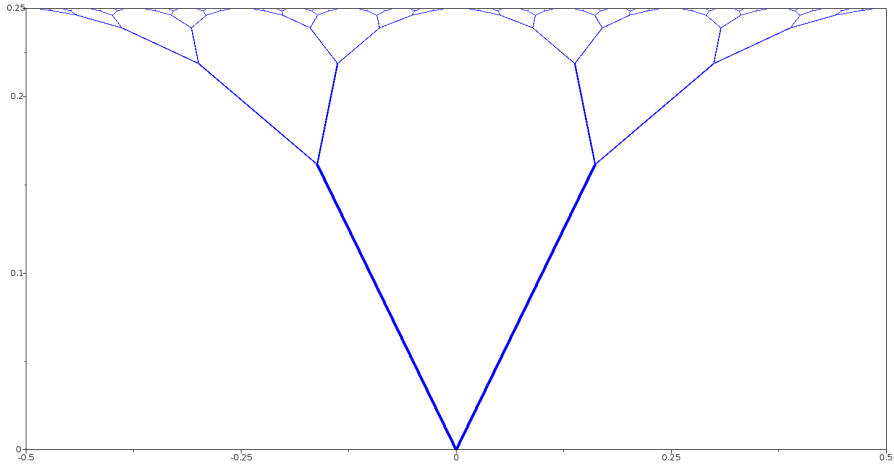}
   \label{figmin}
   \caption{The optimal configuration $\mu^*$}
 \end{center}
 \end{figure}

 \begin{theorem}\label{main}
  For $T=1/4$, $\mu_0=\delta_0$ and $\mu_T=dx\restr[-1/2,1/2]$, $\mu^*$ is the unique minimizer of \eqref{minmainintro}. Moreover, if $T\ge 1/4$, the unique minimizer of \eqref{minmainintro} is given by $\mu_t=\delta_0$ for $t\in[0,T-1/4]$ and
  $\mu_t=\mu^*_{t-(T-1/4)}$ for $t\in(T-1/4,T)$, with
  \[\E(\mu)=\frac{1}{2-\sqrt{2}} +T.\]
  
 \end{theorem}
As a consequence, we obtain the following corollary (see Lemma \ref{lemrescale} for the exact definitions of the rescaling and shear)
 \begin{corollary}\label{cormain}
  For $X\in \R$, $T,\phi>0$ with $T\phi^{-3/2}\ge 1/4$, the unique minimizer of \eqref{minmainintro} with $\mu_0=\phi \delta_X$ and $\mu_T= dx\restr [-\phi/2,\phi/2]$ is given by a suitably sheared and rescaled version of the optimal measure for $X=0$, $\phi=1$ and $\widehat{T}=T\phi^{-3/2}$. Moreover
  \[
   \E(\mu)=\phi^{3/2} \frac{1}{2-\sqrt{2}} +T+\frac{\phi}{T} |X|^2.
  \]

 \end{corollary}
As an application of Corollary \ref{cormain}, we will  further derive a full characterization of symmetric (with respect to $t=0$) minimizers in the case $a=-b=-T$, $\mu_{\pm}=dx\restr [-1/2,1/2]$ and $T\ge 1/4$ (see Theorem \ref{structureTlarge}).\\ 

The proof of Theorem \ref{main} is based on the tree structure of the minimizers of \eqref{minmainintro} (see Proposition \ref{reg}) which together with invariance by scaling and shearing (Lemma \ref{lemrescale}) leads to a recursive characterization of the minimizers. Indeed, if we let 
\[
 E(T):=\min\{ \E(\mu) \ : \ \mu_0=\delta_0, \ \mu_T= dx\restr [-1/2,1/2]\},
\]
then (see \eqref{recursive})
\begin{equation}\label{eq:introrecurE}
 E(T)=\min_{\sum_{i=1}^N \phi_i=1} \sum_{i=1}^N \phi_i^{3/2} E(T\phi_i^{-3/2}) +\frac{1}{12T} \lt(1-\sum_{i=1}^N \phi_i^3\rt).
\end{equation}
This formula reflects the fact that if at level $T$, the minimizer branches into $N$ pieces of respective masses $\phi_i$ then up to rescaling and shearing, each of the subtrees
solves the exact same problem as the original one (connecting a Dirac mass to the Lebesgue measure). In particular, if we define $T_*$ to be the first branching time, meaning that if $T>T_*$ then the minimizer of $E(T)$ cannot branch for a time $T-T_*$, we may use that $T_* \phi_i^{-3/2}\ge T_*$ to obtain 
\[
 E(T_* \phi_i^{-3/2})= E(T_*)+ T_*(\phi_i^{-3/2}-1)
\]
and then rewrite \eqref{eq:introrecurE} in the purely analytical form (see \eqref{alternativeq})
\begin{equation}\label{eq:introalternativeq}
 \frac{E(\Bl)-\Bl}{\Bl}=\min_{\sum_{i=1}^N\phi_i=1} \frac{(N-1)+\frac{1}{12\Bl^2} \lt(1-\sum_{i=1}^N \phi_i^{3}\rt)}{1-\sum_{i=1}^N \phi_i^{3/2}}.
\end{equation}
The idea of the proof of Theorem \ref{main} is to use \eqref{eq:introalternativeq} to prove that $T_*=1/4$ and that the corresponding minimizer has exactly two branches of mass $1/2$ at time zero. Once this is proven, the conclusion is readily
reached thanks to the recursive nature of the problem.\\
In order to prove that $T_*=1/4$ and $N=2$, we introduce for fixed $N\ge 2$ the quantity
\[
 \alpha_N:=\inf_{\phi_i\ge 0}\lt\{ \frac{1-\sum_{i=1}^N \phi_i^3}{1-\sum_{i=1}^N \phi_i^{3/2}} \ : \ \sum_{i=1}^N \phi_i=1\rt\}.  
\]
By \eqref{eq:introalternativeq}, if at time $T_*$ the minimizer has $N$ branches then since for $\sum_{i=1}^N \phi_i=1$ there holds $1-\sum_{i=1}^N \phi_i^{3/2}\le 1-N^{-1/2}$, we have  the lower bound
 \[
  \frac{E(\Bl)-\Bl}{\Bl}\ge  \frac{N-1}{1-N^{-1/2}} +\frac{\alpha_N}{12 \Bl^2}.
 \]
In Proposition \ref{estimTstar}, we use this together with an upper bound on $E(\Bl)$ given by a dyadically branching construction to obtain both that $\Bl\le 1/4$ and that a lower bound on $\alpha_N$ gives a corresponding upper bound on $N$ (see \eqref{criterionalphaN}).
These lower bounds on $\alpha_N$ are obtained in Lemma \ref{lem:Nge3} using a computer assisted proof whose details are given in Appendix \ref{appendix}. This excludes that $N\ge 3$. The case $N=2$ is finally studied in Proposition \ref{prop:Neq2} where we prove that $\Bl=1/4$ and that the mass splits in half. \\

The variational problem \eqref{minmainintro} may be seen as a two dimensional (one for time and one for space) analog of the three dimensional (one for time and two for space) problem derived in \cite{CoGoOtSe}
as a reduced model for the description of branching in type-I superconductors in the regime of very small applied external field. We refer the reader to \cite{CoGoOtSe} for more precise physical motivations and references. In this regime, the natural Dirichlet conditions appearing are $\mu_{\pm}=dx\restr[-1/2,1/2]$. Let us point out that in
the three dimensional model, the term $\sharp\{\phi_i\neq 0\}$ is replaced by $\sum_i \phi_i^{1/2}$. This is in line with the interpretation of the first term in \eqref{functional} as an interfacial
term penalizing the creation of many flux tubes. That is, if we are in $(1+d)-$dimensions, it is proportional to the perimeter of a union of $d-$dimensional balls of volume $\phi_i$ (which is $2\sqrt{\pi}\sum_i \phi_i^{1/2}$ if $d=2$ and $2\sharp\{\phi_i\neq 0\}$ if $d=1$).
As already alluded to, the second term in \eqref{functional} may be interpreted as the Wasserstein transportation cost \cite{AGS, Sant,villani} of moving such balls.\\
In many models describing pattern formation in material sciences, branching patterns similar to the one observed here are expected. However, it is usually very hard to go beyond scaling laws \cite{KohnMuller94,Zwicknagl2014,ChoKoOtmicro,CoOtSer}. In some cases, reduced models have been derived \cite{ViehOtt,CoGoOtSe,CoDiZw} but so far the best results concerning the minimizers are local energy bounds leading
to the proof of asymptotic self-similarity \cite{Conti2006,Viehmanndiss}. Our result is thus the first complete characterization of a minimizer in this context. Of course, this was possible thanks to the simplicity of our model (one dimensional trees in a two dimensional ambient space). We should however point out
that our result is not fully satisfactory since we are essentially able to study only the situation of an isolated microstructure (due to the constraint $T\ge 1/4$) whereas one is typically interested in the
case $T\ll1$ where many microstructures are present and where the lateral boundary conditions have limited effect i.e. one tries to capture an extensive behavior of the system. As detailed in the final section \ref{sec:appli},
we  believe that even in the regime $T\ll 1$, every microstructure is of the type described in Corollary \ref{cormain}.\\

As pointed out in \cite{CoGoOtSe}, the functional \eqref{functional} bears many similarities with so-called branched transport (or irrigation) models \cite{Xia,MSM,BeCaMo} (see also \cite{BrBuS,BraSa} for a formulation reminiscent of our model). Also in this class of problems, there 
has been a strong interest for the possible fractal behavior of minimizers. Besides results on scaling laws \cite{BraWirth}  and fractal regularity \cite{BraSol}, to the best of our knowledge,
the only explicit minimizers exhibiting infinitely many branching points have been obtained in \cite{PaoSteTep} for a Dirac irrigating a Cantor set and in \cite{MaMa} in an infinite dimensional context. 
In particular, the optimal irrigation pattern from a Dirac mass to the Lebesgue measure is currently not known for the classical branched transportation model. One important difference between our model and branched transportation 
is that in our case, minimality does not imply triple junctions nor conditions on the angles between the branches.\\

The organization of the paper is the following. In Section \ref{secnot}, we recall the definition and the basic properties of the functional $\E$. Then, in Section \ref{irrig} we prove Theorem \ref{main}. In the final Section \ref{sec:appli}, we give 
an application of Theorem \ref{main} to the irrigation of the Lebesgue measure by itself and state an open problem. Appendix \ref{appendix} contains the computer assisted computation of $\alpha_N$.\\

\textbf{Notation}
In the paper we will use the following notation. The symbols $\simeq$, $\ges$, $\les$, $\ll$ indicate estimates that hold up to a global constant. For instance, $f\les g$ denotes the existence of a constant $C>0$ such that $f\le Cg$. 
 We denote by $\H^1$ the $1-$dimensional Hausdorff measure. For a Borel measure $\mu$, we will denote by $supp\, \mu$ its support.

\section*{Acknowledgment}
I warmly  thank F. Otto for many useful discussions and constant support as well as A. Lemenant for a careful reading of a preliminary version of the paper. The hospitality of the Max Planck Institute for Mathematics in the Sciences of Leipzig, where part of this
research was carried out is  gratefully acknowledged. Part of this research was funded by the program PGMO of the FMJH through the  project COCA.

 \section{The variational problem and main properties of the functional}\label{secnot}
 In this section we first give a rigorous definition of the energy $\E(\mu)$ and then prove that the boundary value  problem \eqref{minmainintro} has minimizers which are locally given by finite union of straight segments so that the representation \eqref{functional} makes sense.   
 \begin{definition}
For $a<b$ we denote by $\calA_{a,b}$ the set of  pairs of measures 
$\m\ge 0$, $m$ with $m\ll\mu$, satisfying the continuity equation
\begin{equation}\label{conteq0}
 \pa_t \m +\pa_x m=0 \qquad \textrm{in } \R\times(a,b)
\end{equation}
and such that $\m=\m_{t}\otimes d{t}$ where, for a.e. $t\in(a,b)$,
$\m_{t}=\sum_i \phi_i \delta_{X_i}$ for some $\phi_i\ge 0$ and $X_i\in \R$. 
We denote by $\calA_{a,b}^*:=\{\mu: \exists m, (\mu,m)\in\calA_{a,b}\}$ the
set of admissible $\mu$.

Further, we define  $\E:\calA_{a,b}\to[0,\infty]$ by
\begin{equation}\label{Imum}
 \E(\mu,m):=
    \int_{a}^b  \sharp\{supp \, \m_t\} dt + 
   \int_{\R\times(a,b)}
   \left(\frac{dm}{d\m} \right)^2 d\m 
\end{equation}
and  (with abuse of notation) $\E:\calA_{a,b}^*\to[0,\infty]$ by
\begin{equation}\label{Imu}
\E(\m):=\min \{ \E(\mu,m)\ : \ m\ll\m, \ \pa_t \m + \pa_x m=0\}.
\end{equation}
\end{definition}
Equation (\ref{conteq0}) is understood in the sense of distributions (testing with test functions in $C^\infty_c(\R\times(0,T))$). Contrary to \cite{CoGoOtSe}, we use free boundary conditions instead of periodic ones but this makes only minor differences.
In the sequel we will only deal with measures $\mu$ of bounded support. In this case, 
because of  \eqref{conteq0},  $\mu_{t}(\R)$ does not depend on $t$. Let us point out that for such measures, the minimum in (\ref{Imu}) is attained thanks to \cite[Th. 8.3.1]{AGS}.
Moreover, the minimizer is unique by strict convexity of $m\to\int_{\R\times(a,b)} \left(\frac{dm}{d\m} \right)^2 d\m$.  {}Let us also notice that by the Benamou-Brenier formula \cite{AGS}, we have
 for every measure $\mu$, and every $t, t'\in(a,b)$,
\begin{equation}\label{HolderW2}
  W_2^2(\mu_{t}, \m_{t'})\le \E(\mu) |t-t'|,
\end{equation}
where  
the $2$-Wasserstein distance between 
 two measures $\mu$ and $\nu$ of bounded second moment with $\mu(\R)=\nu(\R)$  is defined by
\[W_2^2(\mu,\nu):=\min\lt\{ \int_{\R\times \R} |x-y|^2 \, d\Pi(x,y)  \, : \, \Pi_1=\mu, \ \Pi_2=\nu\rt\},\] 
where the minimum is taken over  measures on  $\R\times \R$ and   $\Pi_1$  and $\Pi_2$ are respectively  the first and  second marginal of $\Pi$.
In particular for  every measure $\mu$ with $\E(\mu)<\infty$, the curve $t\mapsto \mu_{t}$ is H\"older continuous with exponent one half in the space of measures (endowed with the metric $W_2$) and the traces $\mu_{a}$ and $\mu_{b}$
 are well defined.

Given two measures $ \m_{\pm}$ on $\R$ with $\ \mu_+(\R)= \mu_-(\R)$ and bounded support, we are interested in the  variational problem
\begin{equation}\label{limitProb}
 \inf\lt\{ \E(\mu) \ : \ \mu_{a}=  \m_{-}, \ \m_b=\mu_+ \rt\}.
\end{equation}
Let us first notice that if $L>0$ is such that $supp \, \mu_-\cup supp \,  \mu_+\subset [-L/2,L/2]$, then we may restrict the infimum in \eqref{limitProb} to measures satisfying $supp \,  \mu_t\subset [-L/2,L/2]$ for a.e. $t\in (a,b)$.
Indeed, if $\mu$ is admissible with $\mu_t=\sum_i \phi_i \delta_{X_i}$ then  letting $\widetilde{X}_i:=  \min(L/2, |X_i|) sign X_i$ and then $\tilde{\mu}_t:=\sum_i \phi_i \delta_{\widetilde{X}_i}$, we get that $\tilde{\mu}$ is admissible and  has lower energy than $\mu$ (i.e. the energy decreases by projection on $[-L/2,L/2]$).  From now on we will only consider such measures.\\

As in \cite[Prop 5.2]{CoGoOtSe} (to which we refer for the proof), a simple branching construction shows that   any pair of  measures with equal flux may be connected with finite cost.

\begin{proposition}\label{branchingmu}
 For every pair of  measures $ \m_{\pm}$ with $supp \, \mu_{\pm}\subset[-L/2,L/2]$ and $\mu_+(\R)=\mu_-(\R)=\phi$, 
 there is  $\mu\in \calA^*_{a,b}$ such that letting $b-a=2T$,
 $\mu_{a}=\mu_-$, $\mu_b=\mu_+$ and
 \begin{equation*}
  \E(\mu)\les  T+\frac{\phi L^2}{T}. 
 \end{equation*}
 If $\mu_+=\mu_-$, then there is a construction with 
 \begin{equation*}
  \E(\mu)\les   T+T^{1/3} \phi^{1/3} L^{2/3}. 
 \end{equation*}
 \end{proposition}
From this, arguing as in \cite[Prop. 5.5]{CoGoOtSe}, we obtain that

\begin{proposition}\label{existmu}
 For every pair of measures $ \m_{\pm }$ with bounded support and  $ \m_+(\R)= \m_-(\R)$, the infimum in \eqref{limitProb} is finite and attained.
\end{proposition}

We now give some regularity results for minimizers of \eqref{limitProb}. These can be mostly proven as in \cite{CoGoOtSe} so we state them without proof. Let us first recall the notion of subsystem.
\begin{proposition} [Definition of a subsystem]\label{def:subsystem}
Given a point $(X, t) \in [-L/2,L/2]\times(a,b) $ and $\mu\in\calA^*_{a,b}$ with $\E(\mu)<\infty$, there exists a subsystem $\m'$ of $\m $ emanating from $(X,t)$. By this we mean that there exists $\m'$ such that
\begin{itemize} 
\item[(i)] $ \m'\le \m$  i.e. $\m-\m'$ is a positive measure,
\item[(ii)] $\m'_{t}=a\delta_{X}$, where $a=\m_{t}(X)$,
\item[(iii)] if $m$ is such that $\E(\mu)=\E(\mu,m)$, then  \begin{equation*}
\pa_t \m'+ \pa_x \lt(\frac{dm}{d\m} \m'\rt)=0.\end{equation*}
\end{itemize}
In particular, (ii) implies that    $(\m_{t} - \m'_{t}) \perp  \delta_{X}$ in the sense of  the Radon-Nikodym decomposition. We call $\mu^+:=\mu'\restr\R\times (t,b)$ the forward subsystem emanating from $(X,t)$ and $\mu^-:=\mu'\restr\R\times (a,t)$ the backward subsystem emanating from $X$. 
\end{proposition}

\begin{lemma}[No loops] \label{noloop} Let $\m$ be a minimizer for the Dirichlet problem (\ref{limitProb}), $\bar t\in (a,b)$.
Let $X_1$, $X_2$ be two points in the line $\{(x,t):t=\bar{t}\}$.  Let $\m_1$ and $\m_2$ be subsystems  of $\m$ emanating from $(X_1,\bar t)$, resp. $(X_2,\bar t)$. Let $(X_+,t_+)$ 
be a point with $ t_+  > \bar{t}$ and $(X_-,t_-)$ a point with $ t_-<\bar{t}$, and such that $\m_1$ and $\m_2$ both have Diracs at both $X_+$ and $X_-$ with nonzero mass. Then $X_1=X_2$.
\end{lemma}

As in \cite{CoGoOtSe}, a consequence of this lemma is that we  have a representation of the form
\begin{equation}\label{representationmu}
\m= \sum_i \frac{\varphi_i}{\sqrt{1+|\dotX_i|^2}} \, \mathcal{H}^1\restr\Gamma_i   \end{equation}
 where the sum is countable and 
 $\Gamma_i=\{ (X_i(t),t) : t\in [a_i,b_i]\} $ with $X_i$ absolutely continuous and almost everywhere 
 non overlapping.\\

Another consequence is that if there are two levels at which $\m$ is a finite sum
of Diracs, then it is the case for all the levels in between. Since $\E(\mu)<\infty$ implies in particular that $\sharp\{\phi_i\neq0\}<\infty$ for a.e. $t\in (a,b)$, this means that $\mu_t$
is in fact a locally  finite  (in time) sum of Dirac masses and thus, the sum in \eqref{representationmu} is  finite away from the initial and finite time.
\\

For measures which are concentrated on finitely many curves, we have as in \cite[Lem. 5.9]{CoGoOtSe}, a representation formula for $\E(\mu)$.

\begin{lemma}\label{lemmacurves}
 Let $\mu=\sum_{i=1}^N \frac{\varphi_i}{\sqrt{1+|\dotX_i|^2}} \, \mathcal{H}^1\restr \Gamma_i \in\calA^*_{a,b}$ with $\Gamma_i=\{ (X_i(t),t) : t\in [a_i,b_i]\}$ for some  absolutely continuous  curves $X_i$,  disjoint up to the endpoints. 
 Every $\phi_i$ is then constant on $[a_i,b_i]$ and we have conservation of mass. That is, 
 for $z:=(x,t)$, letting 
\begin{align*}\calI^-(z)&:=\{ i\in[1,N] \ : \ t=b_i, \ X_i(b_i)=x\} \\ \calI^+(z)&:=\{ i\in[1,N] \ : \ t=a_i, \ X_i(a_i)=x\},  \end{align*} 
there holds
\begin{equation*}
\sum_{i\in\calI^-(z)} \phi_i=\sum_{i\in\calI^+(z)} \phi_i.
\end{equation*}
  Moreover, $m=\sum_i \frac{\varphi_i}{\sqrt{1+|\dotX_i|^2}}
\dotX_i \, \mathcal{H}^1\restr\Gamma_i$ and 
\begin{equation}\label{Iparticul}
 \E(\mu)=\sum_i \int_{a_i}^{b_i}  1+ \phi_i |\dotX_i|^2 dt.
\end{equation}
  
\end{lemma}
In particular, this proves that for minimizers, formula \eqref{Iparticul} holds (where  the sum is at most countable). By a slight abuse of notation, for such measures we will denote
\[\E(\mu)=\int_{a}^b \sharp\{\phi_i\neq 0\}+ \sum_i \phi_i |\dotX_i|^2 dt.\]

%

We  gather below some properties of the minimizers
\begin{proposition}\label{reg}
A minimizer of the Dirichlet problem (\ref{limitProb}) with boundary conditions
${\mu}_{\pm}$  satisfies
\begin{itemize}
\item[(i)] Each $X_i$ is affine.
\item[(ii)] There is  monotonicity of the traces in the sense that for every $t\in(a,b)$, if $\mu_t=\sum_i \phi_i \delta_{X_i}$ with $X_i$ ordered (i.e. $X_i\le X_{i+1}$) and if $\mu^{i,+}$ is the forward subsystem emanating from $X_i$, then  the traces $\mu_b^{i,+}$ satisfy
$supp \, \mu_b^{i,+}=[x_{i}^+,y_i^+]$ with $y_i^+\le x_{i+1}^+$. The analogous  statement  holds for the backward subsystems.  
\item[(iii)] If $\mu_-=\phi \delta_X$ then $\mu$ has  a tree structure.
\item[(iv)] If ${\mu}_{-}= {\mu}_{+} $, then letting $a=-T$ and $b=T$, there exists a minimizer which is symmetric with respect to the $t=0$ plane. For every such minimizer, the number of Dirac masses at time $t$ is minimal for  $t=0$.
\end{itemize}
    \end{proposition}
    \begin{proof}
     Item (i) follows from fixing the branching points and minimizing in $X_i$. The other points are simple consequences of Lemma \ref{noloop}.
    \end{proof}

    The monotonicity property (ii), is analogous to the monotonicity of optimal transport maps in one space dimension \cite{villani}.

 \section{Irrigation of the Lebesgue measure by a Dirac mass}\label{irrig}
 In this section we consider  \eqref{minmainintro} with $a=0$ and $b=T$, $\mu_-=\phi \delta_X$ and $\mu_+=dx \restr [-\phi/2,\phi/2]$. We will denote
 \[E(T,\phi,X):=\min \{\E(\mu) \ : \ \mu_0=\phi \delta_X , \ \mu_T= dx \restr [-\phi/2,\phi/2]\}.\]
 For simplicity we let $E(T,\phi):=E(T,\phi,0)$ be the energy required to connect the Lebesgue measure to the centered Dirac mass and $E(T):=E(T,1)$. The following lemma shows that  understanding $E(T)$ is enough for understanding $E(T,\phi,X)$.
 \begin{lemma}\label{lemrescale}
  For every $T,\phi,X$, there holds,
  \begin{equation}\label{equationEX}
   E(T,\phi,X)=E(T,\phi)+\frac{1}{T} \phi |X|^2.
  \end{equation}
Moreover, if $\mu_t=\sum_i \phi_i X_i(t)$ is optimal for $E(T,\phi)$, then letting $\widehat{X}_i(t):= (1-\frac{t}{T})X+X_i(t)$, $\hat{\mu}_t=\sum_i \phi_i \widehat{X}_i(t)$ is optimal for $E(T,\phi,X)$.\\

 Furthermore, we have 
\begin{equation}\label{rescaleTphi}
 E(T,\phi)=\phi^{3/2}E(T\phi^{-3/2}).
\end{equation}
In addition, if $\mu_t=\sum_i \phi_i X_i(t)$ is optimal for $E(T\phi^{-3/2})$ then letting $\hat{t}:=\phi^{3/2}t$,  $\widehat{\phi}_i:= \phi \phi_i$,
and $\widehat{X}_i:= \phi X_i(\hat{t})$, then $\hat{\mu}_{\hat{t}}= \sum_i \widehat{\phi}_i \delta_{\widehat{X}_i}$ is optimal for $E(T,\phi)$.\\
\end{lemma}
\begin{proof}
  For $\mu_t=\sum_i \phi \delta_{X_i}$ admissible for $E(T,X)$, we define $\hat{\mu}_t:=\sum_i \phi_i \delta_{\hat{X}_i}$, where $\hat{X}_i(t):=  (1-\frac{t}{T})X+X_i(t)$.  Then, $\hat{\mu}_t$ is admissible for $E(T,\phi,X)$ and 
 \begin{align}\label{Erescale}\E(\hat{\mu})&=\int_0^T \sharp\{\phi_i\neq 0\} +\sum_i \phi_i |\dotX_i-\frac{1}{T}X|^2dt \\
  &=\int_0^T \sharp\{\phi_i\neq 0\} +\sum_i \phi_i |\dotX_i|^2 dt +\frac{|X|^2}{T^2}\int_0^T \sum_i \phi_i dt -2\frac{X}{T} \int_0^T \sum_i \phi_i \dotX_i dt \nonumber.
  \end{align}
  For  $\eps>0$, thanks to Lemma \ref{lemmacurves} and the fact that $X_i(0)=0$, we have 
  \[\int_0^{T-\eps} \sum_i \phi_i \dotX_i dt =\sum_i \phi_i X_i(T-\eps).\]
  Furthermore, testing the weak convergence of  $\mu_t$ to $dx$ as $t\to T$, with the function $x$, we get
  \[\lim_{\eps\to 0} \sum_i \phi_i X_i(T-\eps)=\int_{-\phi/2}^{\phi/2} xdx =0.\]
  Finally, since by H\"older's inequality applied twice and $\sum_i \phi_i=\phi$,
  \[\int_{T-\eps}^T \sum_i \phi_i |\dotX_i|\le \phi^{1/2} \int^T_{T-\eps} (\sum_i \phi_i |\dotX_i|^2)^{1/2}\le \phi^{1/2}\eps^{1/2} \lt(\int_{T-\eps}^T  \sum_i \phi_i |\dotX_i|^2\rt)^{1/2},\]
  we get 
  \[\int_0^T \sum_i \phi_i \dotX_i dt=\lim_{\eps\to 0} \lt(\int_0^{T-\eps}  \sum_i \phi_i \dotX_i dt+ \int_{T-\eps}^T \sum_i \phi_i |\dotX_i| dt\rt)=0.\]
  Combining this with $\sum_i \phi_i=\phi$ and \eqref{Erescale}, we get
 \[\E(\hat{\mu})=\E(\mu)+\frac{1}{T} \phi |X|^2,\]
 from which the first part of the proposition follows noticing that the map $\mu\to \hat \mu$ is one-to-one between admissible measures for $E(T,X)$ and admissible measures for $E(T,\phi,X)$.\\
 The second part follows simply by using the rescaling $\hat{t}:=\phi^{3/2}t$,  $\widehat{\phi}_i:= \phi \phi_i$,
and $\widehat{X}_i:= \phi X_i(\hat{t})$.\\
\end{proof}
\begin{remark}
 Using the same type of rescalings as the one leading to \eqref{rescaleTphi}, it is not hard to prove that $T\to E(T)$ is a continuous function.
\end{remark}

As a consequence of the monotonicity of the support and of the previous lemma, we can derive the following fundamental recursive characterization of $E(T)$.
\begin{lemma}
For every $T>0$,
\begin{equation}\label{recursive}
E(T)=\min_{\sum_{i=1}^N \phi_i=1} \sum_{i=1}^N \phi_i^{3/2} E(T\phi_i^{-3/2}) +\frac{1}{12T} \lt(1-\sum_{i=1}^N \phi_i^3\rt), 
\end{equation}
so that in particular the energy does not change if we reorder the $\phi_i$.
 \end{lemma}
\begin{proof}
 Let $\phi_1,..,\phi_N$ be the fluxes of the branches leaving from the point $(0,0)$ (if it is not a branching point then $N=1$). Up to relabeling, we may assume that the $\phi_i$ are ordered i.e. $\phi_1$ corresponds to the first branch, $\phi_2$ to the second and so on.
By the monotonicity of the traces (Proposition \ref{reg}), the $N$ branches are independent and the mass from the first branch will go to $[-1/2,-1/2+\phi_1]$, the second branch will go to $[-1/2+\phi_1,-1/2+\phi_1+\phi_2]$ and so on. Let $\BXi$ be the centers of the intervals of length $\phi_i$ i.e. $\BXi= -\frac{1}{2}+ \sum_{j<i} \phi_j +\frac{\phi_i}{2}$. We first prove that 
From \eqref{equationEX} and \eqref{rescaleTphi} we have 
\begin{equation}\label{toproverecursive}
E(T)=\min_{\sum_{i=1}^N \phi_i =1} \sum_{i=1}^{N} E(T,\phi_i, \BXi)= \min_{\sum_{i=1}^N \phi_i =1} \sum_{i=1}^{N}\phi_i^{3/2}E(T \phi_i^{-3/2}) +\frac{1}{T} \sum_{i=1}^N \phi_i |\BXi|^2,\end{equation}
so that we are left to prove that  for every $(\phi_i)_{i=1}^N$ with $\sum_{i=1}^N \phi_i=1$, there holds
 \begin{equation}\label{sumequalsum}\sum_{i=1}^N \phi_i |\BXi|^2=\frac{1}{12}\lt(1-\sum_{i=1}^N \phi_i^3\rt).\end{equation}

 We prove this by induction on $N$. For $N=1$, there is nothing to prove. For $N=2$, since $\phi_2=1-\phi_1$, the left-hand side of \eqref{sumequalsum} is equal to 
 \[\phi_1\lt(\frac{-1+\phi_1}{2}\rt)^2+(1-\phi_1)\lt(\frac{\phi_1}{2}\rt)^2=\frac{1}{4}\phi_1(1-\phi_1),\]
 which is equal to the right-hand side of \eqref{sumequalsum}.\\
 Assume now that \eqref{sumequalsum} holds for $N-1$. Let then $\phi=\phi_1+\phi_2$ and $\overline{X}=\overline{X}_1+\frac{\phi_2}{2}$. By the induction hypothesis,
 \[\sum_{i=3}^{N} \phi_i |\BXi|^2 +\phi |\overline{X}|^2=\frac{1}{12}\lt(1-\sum_{i=3}^N \phi^3_i -\phi^3\rt)\]
 so that 
 \[\sum_{i=1}^N \phi_i |\BXi|^2=\frac{1}{12}\lt(1-\sum_{i=1}^N \phi_i^3\rt)+\frac{1}{12}(\phi_1^3+\phi_2^3-\phi^3) +\phi_1|\overline{X}_1|^2+\phi_2|\overline{X}_2|^2-\phi|\overline{X}|^2.\]
 We are thus left to prove that 
 \begin{equation}\label{toprovesum}
  \frac{1}{12}(\phi_1^3+\phi_2^3-\phi^3) +\phi_1|\overline{X}_1|^2+\phi_2|\overline{X}_2|^2-\phi|\overline{X}|^2=0.
 \end{equation}
By definition of $\phi$, $\overline{X}$ and since $\overline{X}_2=\overline{X}_1+\frac{\phi_1+\phi_2}{2}$, we have 
\[
 \frac{1}{12}(\phi_1^3+\phi_2^3-\phi^3)=-\frac{\phi_1\phi_2}{4}(\phi_1+\phi_2)=-\frac{\phi_1\phi_2 \phi}{4}.
\]
Analogously we can compute
\begin{align*}
 \phi_1|\overline{X}_1|^2+\phi_2|\overline{X}_2|^2-\phi|\overline{X}|^2= &  \phi_1|\overline{X}_1|^2+\phi_2|\overline{X}_1+\frac{\phi}{2}|^2-\phi|\overline{X}_1+\frac{\phi_2}{2}|^2\\
 =& \phi_1|\overline{X}_1|^2+\phi_2|\overline{X}_1|^2 +\phi_2\phi \overline{X}_1+\frac{\phi_2\phi^2}{4}-\phi|\overline{X}_1|^2-\phi\phi_2\overline{X}_1-\frac{\phi\phi_2^2}{4}\\
 =&\frac{\phi_2\phi}{4}(\phi-\phi_2)\\
 =&\frac{\phi_1\phi_2\phi}{4}.
\end{align*}
Adding these two equalities we get \eqref{toprovesum} which concludes the proof of \eqref{sumequalsum}.
\end{proof}

Before going further, let us point out that for $T,t>0$ using as test configuration for $E(T+t)$, $\delta_0$ in $[0,t]$ extended by the minimizer of $E(T)$ in $[t,T+t]$, we obtain
\begin{equation}\label{estim1}
 E(T+t)\le E(T)+t.
\end{equation}
This together with \eqref{recursive}, motivates the  introduction of  the largest branching time:
\begin{equation}\label{Tstar}T_*=\inf\{T \ : E(T+t)=E(T)+t, \quad \forall t\ge 0\}.\end{equation}
By definition of $\Bl$ and \eqref{estim1}, we see that for every $\eps>0$, $E(\Bl-\eps)+\eps>E(\Bl)$ which means that every minimizer of \eqref{recursive} for $T=\Bl$
must have $N>1$ branches (there must be branching at time zero).\\
We will also need the following simple lemma.
\begin{lemma}\label{straight}
 Let  $X\in \R$, $T,\phi>0$ with $T\phi^{-3/2}>\Bl$. And let $\mu_t$ be a minimizer for $E(T,\phi,X)$. Letting $X(t):=(1-\frac{t}{T})X$, for $t\in[0,T-\phi^{3/2}\Bl]$ there holds,  $\mu_t=\phi \delta_{X(t)}$.
\end{lemma}
\begin{proof}
 Since $T\phi^{-3/2}>\Bl$, by definition of $\Bl$, in $[0,T\phi^{-3/2}-\Bl]$, every minimizer of $E(T\phi^{-3/2})$ is of the form $\delta_0$.
 Therefore, by \eqref{equationEX} and \eqref{rescaleTphi}, if $X(t):=(1-\frac{t}{T})X$, then for $t\in[0,T-\phi^{3/2}\Bl]$,  $\mu_t=\phi \delta_{X(t)}$.
\end{proof}

We can now state the main result of this section.
\begin{proposition}\label{propTstar}
 We have $\Bl=1/4$ and if $\phi_1,..,\phi_N$ are optimal in \eqref{recursive} for $T=\Bl$, then $N=2$ and $\phi_1=\phi_2=1/2$. Moreover,
 \[E(1/4)-1/4=\frac{1}{2-\sqrt{2}}.\]
\end{proposition}
The proof of this proposition will consist of the remaining part of this section. Before doing so, let us see how it implies Theorem \ref{main}. In the proof, we will use the following notation

\begin{definition}
 For $\mu\in \calA^*_{a,b}$ with $\mu_t=\sum_i \phi_i \delta_{X_i}$ and $X\in \R$, let $S_X(\mu)$ be the measure defined by $(S_X(\mu))_t:=\sum_i \phi_i \delta_{X_i+X}$. 
\end{definition}

\begin{proof}[Proof of Theorem \ref{main}]
 By definition of $\Bl$, for $T\ge 1/4$, we have $E(T)=E(1/4)+(T-1/4)$ and if $\mu$ is a minimizer for $E(T)$, then it coincides with $\delta_0$ in $[0,T-1/4]$ and with (a translated version of) a minimizer for $E(1/4)$ in $[T-1/4,T]$. 
 Therefore, it is enough to prove that for $T=1/4$, the only minimizer of $E(1/4)$ is given by $\mu^*$.\\
 
 Let $\mu$ be such a minimizer and let us prove by induction that $\mu=\mu^*$. Recall first that we defined  $t_k=\frac{1}{4}\lt(1-\left(\frac{1}{2}\rt)^{3k/2}\rt)$. 
 Assume that $\mu_t=\mu^*_t$  for  $t\in[0,t_{k-1}]$ and that $\mu_{t_{k-1}}=2^{-(k-1)}\sum_{i=1}^{2^{k-1}} \delta_{X_i^{k-1}} $ for some  ordered $X_i^{k-1}\in \R$.
 By monotonicity of the support, in $[t_{k-1},1/4]$, each of the forward subsystems $\mu^{+,i}$ emanating from $X_i^{k-1}$ must be of the form  $\mu^{+,i}_t= S_{\BXi^{k-1}}(\mu^{i}_{t-t_{k-1}})$ 
 where $\mu^{i}$ is a minimizer of $E(\frac{1}{4}-t_{k-1},2^{-(k-1)}, X_i^{k-1}-\BXi^{k-1})=E(\frac{1}{4} (2^{-(k-1)})^{3/2},2^{-(k-1)}, X_i^{k-1}-\BXi^{k-1})$. By \eqref{rescaleTphi} and Proposition \ref{propTstar}, every minimizer 
 of $E(\frac{1}{4} (2^{-(k-1)})^{3/2},2^{-(k-1)}, X_i^{k-1}-\BXi^{k-1})$ must branch into two pieces of equal mass. Thus, we can further decompose $\mu^{i}=\mu^{i,1}+\mu^{i,2}$ where
 $\mu^{i,1}=S_{-2^{-(k+1)}}(\nu^{i,1})$ with $\nu^{i,1}$ a minimizer for 
 $E(\frac{1}{4} (2^{-(k-1)})^{3/2},2^{-k}, 2^{-(k+1)}+X_i^{k-1}-\BXi^{k-1})$ and similarly for $\mu_i^2$.
 Let 
 \begin{align*}Y_{2i-1}^k(s)&:= -2^{-(k+1)}+(1-\frac{s}{\frac{1}{4}-t_{k-1}})(X_{i}^{k-1}-\overline{X}_{2i-1}^k)  \qquad \textrm{ and } \\ 
  Y_{2i}^k(s)&:= 2^{-(k+1)}+(1-\frac{s}{\frac{1}{4}-t_{k-1}})(X_{i}^{k-1}-\overline{X}_{2i}^k),
 \end{align*}

 Since $\BXi^{k-1}-2^{-(k+1)}=\overline{X}^k_{2i-1}$ and $\BXi^{k-1}+2^{-(k+1)}=\overline{X}^k_{2i}$, 
 by Lemma \ref{straight}, for $s\in[0,t_k-t_{k-1}]$, $\mu^i_s= 2^{-k} (\delta_{Y_{2i-1}^k(s)}+\delta_{Y_{2i}^k(s)})$ and thus letting
 
 \begin{align*}X_{2i-1}^k(t)&:= \frac{t-t_{k-1}}{\frac{1}{4}-t_{k-1}}(\overline{X}_{2i-1}^k-X_{i}^{k-1})+ X_i^{k-1}  \qquad \textrm{ and } \\ 
  X_{2i}^k(t)&:= \frac{t-t_{k-1}}{\frac{1}{4}-t_{k-1}}(\overline{X}_{2i}^k-X_{i}^{k-1})+ X_i^{k-1},
 \end{align*}
we finally obtain as claimed that for $t\in[t_{k-1},t_k]$,
\[\mu^{+,i}_t=2^{-k} (\delta_{X_{2i-1}^k(t)}+\delta_{X_{2i}^k(t)}).\]

\end{proof}
We may  start investigating the properties of  $\Bl$.
\begin{lemma}
 There holds
 \[0<T_*<\infty.\]
 As a consequence, the infimum in \eqref{Tstar} is attained.
\end{lemma}
\begin{proof}
We first observe that for every $T>0$, by \eqref{HolderW2}
 \begin{equation}\label{estimbelowE}E(T)\ge T+\frac{W_2^2(\delta_0,dx\restr[-1/2,1/2])}{T}.\end{equation}
Let us  prove that $\Bl<\infty$.  Let $T\ge 1$ and $\mu_t$ be a minimizer for $E(T)$. By \eqref{estim1}, for every $T\ge 1$, 
 \[E(T)\le E(1)+(T-1).\]
By the no-loop condition, if $\mu_t$  has its first branching at time $t_0$ then in $[t_0,T]$ it has at least two branches and thus
 \[E(T)\ge 2(T-t_0).\]
 Putting these two inequalities together we get $t_0\ge T-(E(1)-1)$. Letting $T_1:= E(1)-1$, which is positive by \eqref{estimbelowE}, and assuming that $T\ge T_1$, this  implies that before  $T-T_1$, 
 no branching may occur. Hence, for $T\ge T_1$,
 \[E(T)=E(T_1)+ (T-T_1),\]
 that is $\Bl\le T_1$.\\
 We now prove that $\Bl>0$. By \eqref{estimbelowE}, if $\Bl=0$, for every $T_1\le T$,
\[E(T)=E(T_1)+T-T_1\ge T +\frac{W_2^2(\delta_0,dx\restr[-1/2,1/2])}{T_1}\]
which letting $T_1\to 0$ would give a contradiction to $E(T)<\infty$. The fact that the  infimum in \eqref{Tstar} is attained follows by continuity of $T\to E(T)$.
 \end{proof}
 The next result is a form of equipartition of energy which will be used to prove that $T_*\le 1/4$.

\begin{lemma}
 If $\phi_1,..,\phi_N>0$ are  such that
 \[E(\Bl)= \sum_{i=1}^N \phi_i^{3/2} E(\Bl \phi_i^{-3/2})+ \frac{1}{12\Bl} \lt(1-\sum_{i=1}^N \phi_i^3\rt),\]
 then 
 \begin{equation}\label{mainestim}
  \Bl(N-1)= \frac{1}{12\Bl} \lt(1-\sum_{i=1}^N \phi_i^3\rt).
 \end{equation}

\end{lemma}

\begin{proof}
By \eqref{sumequalsum}, if we denote by $\BXi$ the barycenters of the intervals of length $\phi_i$, it is enough to prove that 
\begin{equation}\label{mainestimbis}
\Bl(N-1)=\frac{1}{\Bl}  \sum_{i=1}^N \phi_i |\BXi|^2. 
\end{equation}
   Let $\mu_t=\sum_i\phi_i \delta_{X_i}$ be optimal for $\Bl$. By definition of $\Bl$, the fact that $\Bl\phi^{-3/2}_i>\Bl$ and Lemma \ref{straight}, there is $\beps>0$ such that for $t\in [0,\beps]$, $X_i(t)=\frac{t}{\Bl} \BXi$. 
   For $\beps>\eps>0$, we are going to construct a competitor for $E(\Bl-\eps)$. In $[\beps-\eps,\Bl-\eps]$, let $Y_i(t):=X_i(t+\eps)$ and 
 in $[0,\beps-\eps]$, $Y_i(t):=\frac{1}{\Bl}\frac{\beps}{\beps-\eps} t \BXi$ so that $Y_i(\beps-\eps)=X_i(\beps)$.  Therefore,
 \begin{align*}
  E(\Bl-\eps)&\le E(\Bl)-N\eps +\frac{1}{\Bl^2}\sum_i \phi_i |\BXi|^2 (\beps-\eps) \frac{\beps^2}{(\beps-\eps)^2} - \frac{1}{\Bl^2}\sum_i \phi_i |\BXi|^2 \beps\\
  &= E(\Bl)-\lt(N-\frac{1}{\Bl^2}\sum_i \phi_i |\BXi|^2\rt) \eps +o(\eps^2).
 \end{align*}
Hence, 
\[\frac{E(\Bl-\eps)-E(\Bl)}{-\eps}\ge N-\frac{1}{\Bl^2}\sum_i \phi_i |\BXi|^2 +o(\eps).\]
Using \eqref{estim1} and letting $\eps\to 0$, we get
\[N-\frac{1}{\Bl^2}\sum_i \phi_i |\BXi|^2\le 1.\]
We similarly define a competitor for $E(\Bl+\eps)$ by letting $Y_i(t):=X_i(t-\beps)$ in $[\eps+\beps,\Bl+\eps]$ and $Y_i(t):=\frac{1}{\Bl} \frac{\beps}{\eps+\beps} t \BXi$ in $[0, \eps+\beps]$ and get
 \begin{align*}
  E(\Bl+\eps)&\le E(\Bl)+N\eps +\frac{1}{\Bl^2}\sum_i \phi_i |\BXi|^2 (\beps+\eps) \frac{\beps^2}{(\beps+\eps)^2} - \frac{1}{\Bl^2}\sum_i \phi_i |\BXi|^2 \beps\\
  &= E(\Bl)+\lt(N-\frac{1}{\Bl^2}\sum_i \phi_i |\BXi|^2\rt) \eps +o(\eps^2).
 \end{align*}
From this we infer that
\[\frac{E(\Bl+\eps)-E(\Bl)}{\eps}\le N-\frac{1}{\Bl^2}\sum_i \phi_i |\BXi|^2 +o(\eps).\]
By definition of $\Bl$ and by continuity of $E$, 
\[\frac{E(\Bl+\eps)-E(\Bl)}{\eps}=1\]
and therefore 
\[N-\frac{1}{\Bl^2}\sum_i \phi_i |\BXi|^2\ge 1,\]
which conclude the proof of \eqref{mainestimbis}.
\end{proof}

\begin{remark}
 Notice that \eqref{mainestim} is compatible with $N=2$, $\phi_1=\phi_2=1/2$ and $\Bl=\frac{1}{4}$.
\end{remark}

Using the characterization \eqref{recursive}, we  show another characterization of $E(\Bl)$ which has the advantage of not being recursive anymore.

\begin{proposition}
 There holds
 \begin{equation}\label{alternativeq}
  \frac{E(\Bl)-\Bl}{\Bl}=\min_{\sum_{i=1}^N\phi_i=1} \frac{(N-1)+\frac{1}{12\Bl^2} \lt(1-\sum_{i=1}^N \phi_i^{3}\rt)}{1-\sum_{i=1}^N \phi_i^{3/2}}.
 \end{equation}
Moreover, if $\phi_1,..,\phi_N$ are minimizers for $E(\Bl)$ then they also minimize the right-hand side of \eqref{alternativeq} and vice-versa.
 \end{proposition}
\begin{proof}
Let $\bphi_i$ be the optimal fluxes for \eqref{recursive}. By definition of $\Bl$, we have for every $\phi_i$ with $\sum_i \phi_i=1$ (since $\phi_i^{-3/2}\ge1$)
 \begin{align*}
  E(\Bl)&\le\sum_{i=1}^N \phi_i^{3/2} E(\Bl \phi_i^{-3/2})+ \frac{1}{12\Bl} \lt(1-\sum_{i=1}^N \phi_i^{3}\rt)\\
  &=\sum_{i=1}^N \phi_i^{3/2} (E(\Bl) + (\Bl \phi_i^{-3/2}-\Bl))	 +\frac{1}{12\Bl} \lt(1-\sum_{i=1}^N \phi_i^{3}\rt)\\
  &= E(\Bl) \lt(\sum_{i=1}^N \phi_i^{3/2}\rt) +\Bl \sum_{i=1}^N (1-\phi_i^{3/2}) +\frac{1}{12\Bl} \lt(1-\sum_{i=1}^N \phi_i^{3}\rt).
 \end{align*}
Therefore
\[E(\Bl)-\Bl\le (E(\Bl)-\Bl) \lt(\sum_{i=1}^N \phi_i^{3/2}\rt)+(N-1)\Bl+\frac{1}{12\Bl} \lt(1-\sum_{i=1}^N \phi_i^{3}\rt),\]
and then
\[\frac{E(\Bl)-\Bl}{\Bl}\le\frac{(N-1)+\frac{1}{12\Bl^2} \lt(1-\sum_{i=1}^N \phi_i^{3}\rt)}{1-\sum_{i=1}^N \phi_i^{3/2}}\]
with equality for $\phi_i=\bphi_i$.  
\end{proof}

For $N\ge 2$, we introduce a quantity which will play a central role in our analysis. Let
\[\alpha_N:=\inf_{\phi_i\ge 0}\lt\{ \frac{1-\sum_{i=1}^N \phi_i^3}{1-\sum_{i=1}^N \phi_i^{3/2}} \ : \ \sum_{i=1}^N \phi_i=1\rt\}.\]
We now prove that $T_*\le 1/4$ and that a lower bound on $\alpha_N$ gives an upper bound on $N$.

\begin{proposition}\label{estimTstar}
 There holds 
 \[\Bl\le \frac{1}{4}.\]
 Moreover, the number $N>1$ of branches of the minimizer for $\Bl$, satisfies
 \begin{equation}\label{estimcentral1}
  \sqrt{N}(\sqrt{N}+1) \Bl+\frac{\alpha_N}{12 \Bl}\le \frac{\sqrt{2}}{2(\sqrt{2}-1)}.
 \end{equation}
As a consequence,
\begin{equation}\label{criterionalphaN}
\sqrt{N}\le \frac{1}{2} \lt(-1+\lt(1+\frac{6}{\alpha_N(\sqrt{2}-1)^2}\rt)^{1/2}\rt). 
\end{equation}
In particular, since $\alpha_N\ge 1$, this gives $N\le 6$.
 
 \end{proposition}
\begin{proof}
We start with the upper bound $\Bl\le \frac{1}{4}$. Let $\phi_1,..,\phi_N$ be optimal for $E(\Bl)$. Then, by \eqref{alternativeq} and \eqref{mainestim},
 \[\frac{E(\Bl)-\Bl}{\Bl}=\frac{2(N-1)}{\lt(1-\sum_i \phi_i^{3/2}\rt)}.\]
 Using that for every $\phi_1,.., \phi_N$ with $\sum \phi_i=1$,
 \begin{equation}\label{convex32}\frac{1}{1-\sum_i \phi_i^{3/2}}\ge \frac{1}{1-\sqrt{N}^{-1}}\end{equation}
 we get
 \[\frac{E(\Bl)-\Bl}{\Bl}\ge\frac{2(N-1)\sqrt{N}}{\sqrt{N}-1}.\]
 Since the right-hand side is minimized for $N=2$ (among $N\in\N$, $N\ge 2$), we have
 \begin{equation}\label{lowerbound}
  \frac{E(\Bl)-\Bl}{\Bl}\ge \frac{2\sqrt{2}}{\sqrt{2}-1}.
 \end{equation}
 We proceed further by proving an upper bound for the left-hand side of \eqref{lowerbound}. For every $\l>0$, we can construct the self similar competitor for which every branch is
 divided into two branches of half the mass at every branching point (which is at $\l_k=\l(1-(\frac{1}{2})^{3k/2})$. Let $\tilde{E}(\l)$ be its energy. Then,
arguing as in \eqref{recursive}, we have
\[\tilde{E}(\l)= 2 \lt(\frac{1}{2}\rt)^{3/2}\tilde{E}(\l)+2\lt(\l-\lt(\frac{1}{2}\rt)^{3/2} \l\rt)+\frac{1}{16 \l}\]
that is 
\[(\tilde{E}(\l)-\l)\lt(1-\frac{1}{\sqrt{2}}\rt)= \l +\frac{1}{16 \l}\]
from which we get
\[\frac{\tilde{E}(\l)-\l}{\l}=  \frac{\sqrt{2}}{\sqrt{2}-1}\lt(1+\frac{1}{16 \l^2}\rt).\]
Since by definition $E(\l)\le \tilde{E}(\l)$, we get 
\begin{equation}\label{upperboundE}\frac{E(\l)-\l}{\l}\le  \frac{\sqrt{2}}{\sqrt{2}-1}\lt(1+\frac{1}{16 \l^2}\rt).\end{equation}
For $\l> \frac{1}{4}$, the right-hand side is strictly smaller than $\frac{2\sqrt{2}}{\sqrt{2}-1}$ hence by \eqref{lowerbound}, we cannot have $\l=\Bl$. This gives the upper bound.\\
 
 We now turn to \eqref{estimcentral1}. For this, we notice that
 \begin{equation}\label{estimcentral2}
  \frac{E(\Bl)-\Bl}{\Bl}=  \frac{N-1}{1-\sum_i \phi_i^{3/2}}+\frac{(1- \sum_i \phi_i^3)}{12\Bl^2(1-\sum_i \phi_i^{3/2})}\ge \frac{N-1}{1-N^{-1/2}}+\frac{\alpha_N}{12 \Bl^2}.
 \end{equation}
Since $\Bl\le 1/4$, by definition of $\Bl$ (recall \eqref{Tstar}), $E(1/4)-1/4=E(\Bl)-\Bl$, combining \eqref{estimcentral2} with \eqref{upperboundE} for $T=1/4$ and 
\[
 \frac{N-1}{1-N^{-1/2}}=\sqrt{N}(\sqrt{N}+1),
\]
 yields \eqref{estimcentral1}.
We finally derive \eqref{criterionalphaN}. For this, multiply \eqref{estimcentral1}  by $\Bl$ to obtain that 
\[ 
\sqrt{N}(\sqrt{N}+1) \Bl^2-\frac{\sqrt{2}}{2(\sqrt{2}-1)} \Bl +\frac{\alpha_N}{12}\le 0. 
\]
This implies that the polynomial $\sqrt{N}(\sqrt{N}+1) X^2-\frac{\sqrt{2}}{2(\sqrt{2}-1)} X +\frac{\alpha_N}{12}$ has real roots (and that $\Bl$ lies between these two roots) so that 
\[\Delta=\frac{1}{2(\sqrt{2}-1)^2}-\frac{\alpha_N}{3}\sqrt{N}(\sqrt{N}+1)\ge 0,\]
which is equivalent to 
\[(\sqrt{N})^2+\sqrt{N}-\frac{3}{2\alpha_N(\sqrt{2}-1)^2}\le 0.\]
Since the largest root of this polynomial (in the variable $\sqrt{N}$) is given by $\frac{1}{2} (-1+(1+\frac{6}{\alpha_N(\sqrt{2}-1)^2})^{1/2})$, we have obtained \eqref{criterionalphaN}.
\end{proof}

\begin{remark}
 From the proof of the previous proposition, one could also get a lower bound for $\Bl$. We will make use of this fact later on to study the case $N=2$.
\end{remark}

Estimate \eqref{criterionalphaN} shows that by obtaining a good lower bound on $\alpha_N$, we may exclude that $N\ge 3$. This is the purpose of the next lemma whose proof is essentially postponed to Appendix \ref{appendix}.
\begin{lemma}\label{lem:Nge3}
 For $3\le N\le 6$,
 \begin{equation}\label{criterionalphaNbis}
 \alpha_N>\frac{6}{(\sqrt{2}-1)^2 (2\sqrt{N}+1)^2}.
\end{equation}
As a consequence, the number of branches at $T=\Bl$  equals two.
\end{lemma}
\begin{proof}

By inverting the relation between $\alpha_N$ and $N$ in \eqref{criterionalphaN}, it is readily seen that \eqref{criterionalphaNbis} and \eqref{criterionalphaN} are incompatible. 
Therefore, proving the lower bound \eqref{criterionalphaNbis} directly excludes the possibility of having $N>2$ branches. Since the proof of \eqref{criterionalphaNbis} 
is basically based on a reduction of the problem defining $\alpha_N$ to a union of one dimensional optimization problems which can be solved by a computer assisted proof, we postpone it to Appendix \ref{appendix}.  
\end{proof}

We now conclude the proof of Proposition \ref{propTstar} by studying the case $N=2$. \\
 
 Let first
\[T_-:=\frac{1}{4}(1-(1-\frac{4}{3}(2-\sqrt{2})^{1/2}).
\]
Then $1/4\ge \Bl\ge T_-$. Indeed, \eqref{estimcentral1} for $N=2$ (recalling that $\alpha_2=2$) may be seen to be equivalent to 
\begin{equation}\label{quadraticT}\Bl^2-\frac{1}{2}\Bl+\frac{1}{6\sqrt{2}(\sqrt{2}+1)}\le 0.\end{equation}
In particular, $\Bl$ has to lie between the two roots of the right-hand side of \eqref{quadraticT} which yields $\Bl\ge T_-$ (recall that the bound $1/4\ge \Bl$ was already derived in Proposition \ref{estimTstar}).\\

  Now, if $\Bl\le 1/4$ then $E(1/4)=E(\Bl)+1/4-\Bl$ so that using \eqref{upperboundE} for $T=1/4$, we get
  \[
   E(\Bl)-\Bl=E(1/4)-1/4\le \frac{\sqrt{2}}{2(\sqrt{2}-1)}.
  \]
 Using \eqref{alternativeq}, we obtain  for $\phi$ optimal for $\Bl$, the bound 
\begin{equation}\label{boundalmostend}
 \Bl \frac{1+\frac{1}{4\Bl^2} \phi(1-\phi)}{1-\phi^{3/2}-(1-\phi^{3/2})}\le  \frac{\sqrt{2}}{2(\sqrt{2}-1)}.
\end{equation}
The next proposition shows the reverse inequality which concludes the proof of Proposition~\ref{propTstar}.

\begin{proposition}\label{prop:Neq2}
 For $T\in [T_-,1/4]$ and $\phi\in [0,1]$,
 \begin{equation}\label{tofinishtheproof}
  T \frac{1+\frac{1}{4T^2}\phi(1-\phi)}{1-\phi^{3/2}-(1-\phi)^{3/2}}\ge \frac{\sqrt{2}}{2(\sqrt{2}-1)},
 \end{equation}
 with equality if and only if $\Bl=\frac{1}{4}$ and $\phi=1/2$.
\end{proposition}
 \begin{proof}
  By symmetry we may assume that $\phi\in[0,1/2]$. Letting $\lambda:= 2T$, \eqref{tofinishtheproof} is equivalent to show that for $\lambda\in [2T_-,1/2]$ and $\phi\in [0,1/2]$,
\begin{equation}\label{finishreduced}
 1+\frac{1}{\lambda^2}\phi(1-\phi)\ge \frac{\sqrt{2}}{\lambda(\sqrt{2}-1)} (1-\phi^{3/2}-(1-\phi)^{3/2}).
\end{equation}
It will be more convenient to work with $a:= \frac{3\sqrt{2}\lambda}{2(\sqrt{2}-1)}$. Letting 
\[a_-:=\frac{3\sqrt{2}}{\sqrt{2}-1}T_-=\frac{3}{2(2-\sqrt{2})}(1-(1-\frac{4}{3}(2-\sqrt{2}))^{1/2})\simeq 1.36 ,\] we are reduced to $a \in[a_-,\frac{3\sqrt{2}}{4(\sqrt{2}-1)}]$. Inequality \eqref{finishreduced} then reads
\begin{equation}\label{finishreduced2}
 L(a,\phi):=1+\frac{9}{2a^2(\sqrt{2}-1)^2}\phi(1-\phi)\ge \frac{3}{a(\sqrt{2}-1)^2} (1-\phi^{3/2}-(1-\phi)^{3/2})=:R(a,\phi).
\end{equation}
Let us first notice that $L(a,0)=1>0=R(a,0)$ and that for $\phi=\frac{1}{2}$, \eqref{finishreduced2} reads,
\[1+\frac{9}{8a^2(\sqrt{2}-1)^2}\ge \frac{3}{a\sqrt{2}(\sqrt{2}-1)},\]
which always holds true (in terms of $\lambda$, this amounts to $1+\frac{1}{4 \lambda^2}\ge \frac{1}{\lambda}$). Moreover, the inequality above is strict if $a<\frac{3\sqrt{2}}{4(\sqrt{2}-1)}$. We are going to study the variations (for fixed $a$) of $L(a,\phi)-R(a,\phi)$.
By differentiating, this is equivalent to study the sign of  
\begin{equation}\label{finishderived}
 D(\phi):= 1-2\phi- a( (1-\phi)^{1/2}-\phi^{1/2}). 
\end{equation}
Let  $X:=\phi^{1/2}$. For $a\in [a_-,\frac{3\sqrt{2}}{4(\sqrt{2}-1)}]$ and  $X\in[0,1/\sqrt{2}]$, since $1-2X^2+aX\ge 0$,
the sign of \eqref{finishderived} is the same as the sign of 

\[P(X):=(1-2X^2+aX)^2- a^2(1-X^2)=4X^4 -4a X^3+2(a^2-2)X^2+2aX+(1-a^2).\]

Since $P$ has roots $\{\pm 1/\sqrt{2}\}$, we can factor it to obtain
\[P(X)= 2(X^2-\frac{1}{2})(2X^2-2aX +(a^2-1)).\] 
For $a>\sqrt{2}$, $2X^2-2aX +(a^2-1)$ has no real roots and therefore, for $a\in[\sqrt{2},   \frac{3\sqrt{2}}{4(\sqrt{2}-1)}]$, $P$ is negative inside $[0,1/\sqrt{2}]$ and thus $\partial_\phi L- \partial_\phi R\le 0$ implying that 
\begin{equation}\label{biggersqt}\min_{\phi\in[0,1/2]} L(a,\phi)-R(a,\phi)=L(a,1/2)-R(a,1/2)\ge 0,\end{equation}
 with strict inequality if $a< \frac{3\sqrt{2}}{4(\sqrt{2}-1)}$ or $\phi\neq 1/2$. This proves \eqref{finishreduced} for $a\in[\sqrt{2},   \frac{3\sqrt{2}}{4(\sqrt{2}-1)}]$. If now $a\in [a_-,\sqrt{2}]$, besides $\pm 1/\sqrt{2}$, $P$ has two more roots
\begin{equation}\label{rootsP}X_\pm:=\frac{a\pm\sqrt{2-a^2}}{2}.\end{equation}
For $a\in[a_-,\sqrt{2}]$,
\[0\le X_-\le 1/\sqrt{2}\le X_+,\]
 and thus  $P$ is negative in $[0,X_-]$ and positive in $[X_-,1/\sqrt{2}]$ from which,
\begin{equation}\label{Psi}\Psi(a):=\min_{\phi\in[0,1/2]} L(a,\phi)-R(a,\phi)=L(a,X_-^2)-R(a,X_-^2).\end{equation}
Let us now prove that for $a\in[a_-,\sqrt{2})$, $\Psi'(a)\le 0$. We first compute
\[\Psi'(a)=\partial_a L(a,X_-^2)-\partial_a R(a,X_-^2)+ 2X_-\partial_a X_-(\partial_\phi L(a,X_-^2)- \partial_\phi R(a,X_-^2)).\]
 By minimality of $X_-$,  $\partial_\phi L(a,X_-^2)- \partial_\phi R(a,X_-^2)=0$ so that 
 \[
 \Psi'(a)=\partial_a L(a,X_-^2)-\partial_a R(a,X_-^2)=\frac{3}{a^2(\sqrt{2}-1)^2}\lt(1-X_-^3-(1-X_-^2)^{3/2}-\frac{3}{a} X_-^2(1-X_-^2)\rt). 
 \]
A simple computation shows that $X_-^2=\frac{1}{2}(1-2a\sqrt{2-a^2})$ so that $\Psi'\le 0$ is equivalent to
\[\frac{1}{2\sqrt{2}}(1-2a\sqrt{2-a^2})^{3/2}+\frac{1}{2\sqrt{2}}(1+2a\sqrt{2-a^2})^{3/2}+\frac{3}{4a}(1-a^2(2-a^2))\ge 1.\]
This indeed holds since for $a\in[a_-,\sqrt{2}]$,
\begin{multline*}
 \frac{1}{2\sqrt{2}}(1-2a\sqrt{2-a^2})^{3/2}+\frac{1}{2\sqrt{2}}(1+2a\sqrt{2-a^2})^{3/2}+\frac{3}{4a}(1-a^2(2-a^2))\ge\\
 \frac{1}{2\sqrt{2}}( (1-2a_-\sqrt{2-a_-^2})^{3/2}+1)+\frac{3}{4\sqrt{2}}(1-a_-^2(2-a_-^2)) \simeq 1.02>1.
\end{multline*}
Therefore, $\Psi'\le 0$ and thus for $a\in[a_-,\sqrt{2}]$, by \eqref{biggersqt}
\[\Psi(a)\ge \Psi(\sqrt{2})> 0,\]
which ends the proof of \eqref{tofinishtheproof}.
 
 \end{proof}

 \begin{remark}
  From \eqref{quadraticT}, one could infer the simpler bound $\Bl\ge \frac{1}{3\sqrt{2}(\sqrt{2}+1)}$ which leads to $a\ge 1$. For $a\in[1,\sqrt{2}]$, we still have \eqref{rootsP} and \eqref{Psi}.
  Numerically, it seems that $\Psi$ is decreasing not only in $[a_-,\sqrt{2}]$ but actually on the whole $[1,\sqrt{2}]$. We were unfortunately not able to prove this fact which would have yield a more elegant proof of \eqref{tofinishtheproof}. 
 \end{remark}

\section{Applications and open problems}\label{sec:appli}
In this section we use Theorem \ref{main} to characterize the symmetric minimizers of 
\begin{equation}\label{symmproblem}\min\{ \E(\mu) \ : \ \mu_{\pm T}=\phi/L dx\restr [-L/2,L/2]\},\end{equation}
at least for $T$ large enough. By rescaling, it is enough to consider $\phi=L=1$.
\begin{theorem}\label{structureTlarge}
 For $T\in[\frac{1}{4},\frac{1}{4(2\sqrt{2}-2)})$, the unique symmetric minimizer of \eqref{symmproblem} is equal in $[0,T]$ to $S_{-1/4}(\mu^1)+S_{1/4}(\mu^1)$, where $\mu^1$  is equal to the unique minimizer of $E(T,1/2)$ given by Corollary \ref{cormain} in $[0,T]$ and to its symmetric in $[-T,0]$. 
 For $T> 4(2\sqrt{2}-2)$, it is given by the unique minimizer of $E(T)$ given by Theorem \ref{main} in $[0,T]$ and to its symmetric in $[-T,0]$. 
\end{theorem}
\begin{proof}
 By Proposition \ref{reg}, we know that a symmetric minimizer exists. Let $\mu$ be such a minimizer. Thanks to the symmetry, we can restrict ourselves to study its structure in $[0,T]$. We let 
 \[
 \E^+(\mu):=  \int_{0}^T  \sharp\{supp \, \m_t\}  + \sum_i \phi_i |\dotX_i|^2 dt.
 \]
Let $\mu_0=\sum_{i=1}^N \phi_i \delta_{X_i}$. We first claim that $X_i=\BXi$, where as before, $\BXi=\frac{-1}{2}+\sum_{j<i} \phi_j +\frac{\phi_i}{2}$. 
 Indeed, applying the same shear as in \eqref{equationEX}, we obtain by minimality of $\mu$,
 \[\E^+(\mu)\ge  \E^+(\hat{\mu})+\frac{1}{T}\sum_{i=1}^N \phi_i |X_i-\BXi|^2\ge \E^+(\mu)+\frac{1}{T}\sum_{i=1}^N \phi_i |X_i-\BXi|^2,\]
 where the inequality could arise from a decreasing of the number of branches after the shear. This proves the claim. For $i=1, .., N$, let $\mu^{+,i}$ be the forward system emanating from $(X_i,0)$. Then, by monotonicity of the traces (Proposition \ref{reg}),
 $\mu^{+,i}_T=dx\restr[\BXi-\phi_i/2,\BXi+\phi_i/2]$ and thus, by the no-loop property
 \begin{equation}\label{splittingEplus}\E^+(\mu)=\sum_{i=1}^N\E^+(\mu^{+,i})=\sum_{i=1}^N E(T,\phi_i).\end{equation}
 Moreover,  $\mu^{+,i}= S_{\BXi}(\mu^i)$ where $\mu^i$ is some minimizer of $E(T,\phi_i)$. Let now $T\ge 1/4$. Since $\phi_i\le 1$, we have $\phi_i^{-3/2} T\ge 1/4$ and thus by Corollary \ref{cormain},
 \[
 \E^+(\mu)=\frac{1}{2-\sqrt{2}}\sum_{i=1}^N \phi_i^{3/2} +NT.  
 \]
For fixed $N$, this is minimized by $\phi_i=1/N$ so that 
\[
 \E^+(\mu)=\frac{1}{2-\sqrt{2}} N^{-1/2} +NT.
\]
The function $x\to \frac{1}{2-\sqrt{2}} x^{-1/2} +xT$ is minimized by $x_{opt}=\lt(2T(2-\sqrt{2})\rt)^{-2/3}$. Since $x_{opt}< 2$ for $T\ge \frac{1}{4(2\sqrt{2}-2)}$ and $3>x_{opt}>2$ for $\frac{1}{4}\le T< \frac{1}{4(2\sqrt{2}-2)}$, this concludes the proof.
\end{proof}
As already explained in the introduction, this theorem is not completely satisfactory. Indeed, physically, the most significant case is $T\ll 1$ (where many microstructures should appear), 
which is not covered by Theorem \ref{structureTlarge}. However, if we could prove the following conjecture,\\

\smallskip
\noindent \textbf{Conjecture}\\
For $T<\Bl$ every minimizer of $E(T)$ branches at time zero (or equivalently $E(T-\eps)<E(T)-\eps$ for $\eps$ small enough),\\

\smallskip
 \noindent then the picture would be almost complete.  Indeed, in that case, arguing as in the proof of Theorem \ref{structureTlarge}, we would have that every symmetric minimizer $\mu$ with $\mu_0=\sum_{i=1}^N \phi_i \delta_{X_i}$ would satisfy  \eqref{splittingEplus}. Now for  $1\le i\le N$, let $\phi_{i,1},..,\phi_{i,N_i}$
be the $N_i$ branches starting from $(0,X_i)$. As in the proof of \eqref{recursive}, we would have
\[
  E(T,\phi_i)=\sum_{k=1}^{N_i} E(T,\phi_{i,k},\overline{X}_{i,k}),
\]
where $\overline{X}_{i,k}:= -\frac{1}{2}+\sum_{j<k} \phi_{i,j} +\frac{\phi_{i,k}}{2}$. Since the minimizer corresponding to $E(T,\phi_{i,k},\overline{X}_{i,k})$ cannot branch at time zero, we would have (if the conjecture holds) that $\phi_{i,k}^{-3/2}T\ge 1/4$ 
so that Corollary \ref{cormain} applies and the structure of the minimizers would be fully determined.  Let us point out that our conjecture would be for instance implied by the convexity of $T\to E(T)$.
 
 \appendix
 \section{}\label{appendix}
 We finally prove \eqref{criterionalphaNbis}.
 \begin{lemma}
 For $3\le N\le 6$,
 \begin{equation}\label{criterionalphaNbisappendix}
 \alpha_N>\frac{6}{(\sqrt{2}-1)^2 (2\sqrt{N}+1)^2}.
\end{equation}
\end{lemma}
 \begin{proof}

In Table \ref{critictable}, the  values  of the right-hand side of \eqref{criterionalphaNbisappendix} are given for $N\in[3,6]$. Since we are not able to prove these bounds analytically, we resort to a computer assisted proof. 
From Table \ref{critictable}, we see that we want  to compute $\alpha_N$ with a precision of $10^{-2}$.  \\

\begin{table}
\begin{center}
\begin{tabular}{|c|c|c|c|c|}
  \hline
 $N$ & 3 & 4&5&6 \\ 
  \hline
  $\alpha_N>$ & 1.76 & 1.4 & 1.17& 1.01\\
  \hline 
\end{tabular}
\caption{Critical values of $\alpha_N$.}
\label{critictable}
\end{center}
\end{table}

\textit{Step 1  (Computation of $\alpha_2$)}:

We start by computing $\alpha_2$. We claim that 
\begin{equation}\label{alpha2}
 \alpha_2=2.
\end{equation}
Since 
\[\alpha_2=\min_{\phi\in[0,1/2]} \frac{3\phi(1-\phi)}{1-\phi^{3/2}-(1-\phi)^{3/2}},\]
in order to prove that the minimum is attained for $\phi=0$, it is enough to show that for $\phi\in [0,1/2]$,
\[\phi(1-\phi)\ge \frac{2}{3}(1-\phi^{3/2}-(1-\phi)^{3/2}).\]
Since the two expressions agree for $\phi=0$, by differentiating, we are left with the proof of
\[1-2\phi\ge (1-\phi)^{1/2}-\phi^{1/2}\]
or equivalently of
\[1-2\phi+\phi^{1/2}\ge (1-\phi)^{1/2}.\]
 Squaring both sides (notice that $1-2\phi+\phi^{1/2}\ge 0$), this amounts to prove that for $\phi\in[0,1/2]$,
\begin{equation}\label{polyphi}
2\phi^2-2\phi^{3/2}-\phi+\phi^{1/2}\ge0.\end{equation}
Since the polynomial,
\[P(X)=2X^4-2X^3-X^2+X\]
has roots $\{-1/\sqrt{2},0,1/\sqrt{2},1\}$, it is positive in $[0,1/\sqrt{2}]$ and thus considering $X=\phi^{1/2}$, we see that  \eqref{polyphi} holds and thus \eqref{alpha2} is proven.

\medskip
\textit{Step 2 (Lower bound for $\alpha_N$)}:
Consider now $N\ge 3$. Since $\frac{1-\sum_i \phi_i^3}{1-\sum_i \phi_i^{3/2}}$ is continuous on the compact convex set $K=\{0\le \phi_i\le 1 \ : \ \sum_i \phi_i=1\}$ 
(the only problem could arise when all the $\phi_i$'s but one go to zero but then it is easy to see that  $\frac{1-\sum_i \phi_i^3}{1-\sum_i \phi_i^{3/2}}\to 2$), the minimum is attained. If the minimum is attained at the boundary then $\alpha_N=\alpha_{N-1}$. Since the values in Table \ref{critictable} are
decreasing with $N$ and since $\alpha_2=2$, in that case by a simple induction we would be over.
Otherwise, we claim  that for every $N$, the optimal $\phi_i$ may take only two values: $\phi$ repeated $k\in[1,N]$ times and $\frac{1-k\phi}{N-k}$ repeated $N-k$ times. Indeed, fix $i\neq j \in [1,N]$ 
and  for $|\eps|$ small enough, define 
$\hat{\phi}_i:= \phi_i+\eps$, $\hat{\phi}_j:=\phi_j-\eps$ and $\hat{\phi}_k:=\phi_k$ for $k\neq i,j$. By minimality, the derivative in zero of 
\begin{equation}\label{var}\eps \to \frac{1-\sum_k \hat{\phi}_k^3}{1-\sum_k \hat{\phi}_k^{3/2}}\end{equation}
is equal to zero. From this, we get the condition
\[\phi_i^2-\frac{\alpha_N}{2} \phi_i^{1/2}=\phi_j^2-\frac{\alpha_N}{2} \phi_j^{1/2} \qquad \forall i\neq j.\]
Letting $P_N(X):= X^4-\frac{\alpha_N}{2}X$, this means that for every $i\neq j$, $P_N (\phi_i^2)=P_N(\phi_j^2)$. Since $P'_N$ has only one positive root, this means that $\phi_i$ may take at most two values, proving the claim.
We now claim that we can further reduce to $k=N-1$ and $\phi\le 1/N$ i.e. $N-1$ 'small' fluxes and a 'large' one. Indeed, if $\phi_i$ takes the values $\phi<\psi:=\frac{1-k\phi}{N-k}$, 
then from the discussion above, we must have $P_N'(\phi^2)<0<P_N'(\psi^2)$, that is $\phi^{3/2}< \alpha_N/8<\psi^{3/2}$. However, if $k<N-1$, we can use that for $\phi_i=\phi_j=\psi$, the second order derivative of
\eqref{var} is positive at $\eps=0$ to obtain that $\psi^{3/2}\le \alpha_N/8$ which would give a contradiction. \\

Therefore,
\begin{equation}\label{probminalphaN}
 \alpha_N= \min_{\phi\in[0,N^{-1}]} f_{N}(\phi),
\end{equation}
where
\[f_{N}(\phi):=\frac{1-(N-1)\phi^{3}-(1-(N-1)\phi)^3}{1-(N-1)\phi^{3/2}-(1-(N-1)\phi)^{3/2}}.\]
We are thus left to estimate a finite number of one dimensional functions. \\  

For $3\le N\le 6$, we want to estimate  $ \min_{\phi\in[0,N^{-1}]} f_{N}(\phi)$ with a precision of $10^{-2}$. For this we will compute the values of $f_{N}$
for a sufficiently fine discretization of $[0,N^{-1}]$. Let $I\subset[0,1/N]$  and let $\Lambda:= \sup_{\phi\in I} |f'_{N}(\phi)|$. Since for $\phi,\psi\in I$,
\[|f_{N}(\phi)-f_{N}(\psi)|\le \Lambda |\phi-\psi|,\]
in order to get a precision of $10^{-2}$ on $\inf_I f_{N}$, it is enough to use a discretization step $\delta\le 10^{-2} \Lambda^{-1}$. We are thus naturally led to estimate $\sup |f'_{N}|$. 
This is a tedious but rather elementary computation which we include for completeness. 
We can compute
\begin{multline}
f'_{N}(\phi)= \frac{3(N-1)}{1- (N-1)\phi^{3/2}-(1-(N-1)\phi)^{3/2}}\lt( -\phi^2+ (1-(N-1)\phi)^2\rt. \label{fprime}\\ 
\lt.-\frac{1}{2}(-\phi^{1/2}+(1-(N-1)\phi)^{1/2}) f_{N} (\phi)\rt). 
\end{multline}
Since, $\sup_{\phi\in[0,1/N]}|f'_{N}(\phi)|=\lim_{\phi\to 0}  |f'_{N}(\phi)|=\infty$, we need to be a little careful. A  Taylor expansion shows that 
\[\lim_{\phi\to 0} f_{N}(\phi)=2,\]
which is always strictly bigger than the critical values computed in Table \ref{critictable}.
Hence, if we can find $\eta>0$, such that in $[0,\eta]$,  $f'_{N}$ is positive, we will have 
\[\alpha_N=\min\lt( \min_{\phi\in[0,\eta]} f_{N}(\phi),\min_{\phi\in[\eta,1/N]}f_{N}(\phi)  \rt)= \min(2, \min_{\phi\in[\eta,1/N]}f_{N}(\phi)).\]
From \eqref{fprime}, we see that $f'_{N}$ is of the same sign as
 \[ h_N(\phi):=-\phi^2+ (1-(N-1)\phi)^2+\frac{1}{2}(\phi^{1/2}-(1-(N-1)\phi)^{1/2}) f_{N} (\phi).\]
 Since for every $\phi\ge 0$,
 \[-\phi^2+ (1-(N-1)\phi)^2\ge 1-2(N-1)\phi,\]
 and
 \[-(1-(N-1)\phi)^{1/2}\ge -1,\]
 we can first infer that 
 \begin{equation}\label{estimh}h_N(\phi)\ge 1-2(N-1)\phi+\frac{1}{2}(\phi^{1/2}-1) f_{N} (\phi).\end{equation}
 
We may now bound from above $f_{N}(\phi)$. Notice  that regarding the numerator, 
\begin{equation}\label{boundabovefden}
 1-(N-1)\phi^3-(1-(N-1)\phi)^3\le 3(N-1)\phi+(N-1)^3\phi^3.
\end{equation}
In order to bound from below the denominator, we first claim that for $\phi\le \lt(\frac{8}{1+16(N-1)}\rt)^2$,
\begin{equation}\label{bounddenomaux}
 (1-(N-1)\phi)^{3/2}\le 1-\frac{3}{2}(N-1)\phi +\frac{(N-1)}{4} \phi^{3/2}.
\end{equation}
Indeed,  letting
\[\Psi_N(\phi):=(1-(N-1)\phi)^{3/2}-1+\frac{3}{2}(N-1)\phi -\frac{(N-1)}{4} \phi^{3/2},\]
we have $\Psi_N(0)=0$ and 
\[\Psi'_N(\phi)= \frac{3}{2}(N-1)\lt( 1-\frac{1}{4}\phi^{1/2}  -(1-(N-1) \phi)^{1/2}\rt).\]
Since for $\phi\le \lt(\frac{8}{1+16(N-1)}\rt)^2$, $1-\frac{1}{4}\phi^{1/2}  -(1-(N-1) \phi)^{1/2}\le 0$, we have $\Psi'_N\le 0$ and thus also $\Psi_N\le 0$, proving \eqref{bounddenomaux}. 
From this we get that for $\phi\le \lt(\frac{8}{1+16(N-1)}\rt)^2$,
\begin{equation}\label{bounddenom}
 1-(N-1)\phi^{3/2}- (1-(N-1)\phi)^{3/2}\ge \frac{3}{2}(N-1)\phi -\frac{5}{4}(N-1) \phi^{3/2}.
\end{equation}
Putting \eqref{boundabovefden} and \eqref{bounddenom} together, we get that for  $\phi\le \lt(\frac{8}{1+16(N-1)}\rt)^2$,
\begin{equation}\label{boundf}
 f_{N}(\phi)\le 2 \frac{1+\frac{(N-1)^2}{3} \phi^2}{1-\frac{5}{6} \phi^{1/2}}.
\end{equation}

Inserting this back into \eqref{estimh}, we get that for $\phi\le \lt(\frac{8}{1+16(N-1)}\rt)^2$,
\[
 h_N(\phi)\ge 1-2(N-1)\phi+(\phi^{1/2}-1)\frac{1+\frac{(N-1)^2}{3} \phi^2}{1-\frac{5}{6} \phi^{1/2}}.
\]
Since for $\phi\le \lt(\frac{8}{1+16(N-1)}\rt)^2$, $5/6 \phi^{1/2} \le 2/3$ and thus $\frac{\phi^2 (N-1)^2}{3(1- \frac{5}{6} \phi^{1/2})}\le\phi^2(N-1)^2\le  \phi (N-1)$, this may be further simplified into
\[h_N(\phi)\ge 1-3(N-1)\phi+(\phi^{1/2}-1)\frac{1}{1-\frac{5}{6} \phi^{1/2}}.\]
Since for $x\le 1/7$, 
\[\frac{1}{1-x}\le 1+\frac{7}{6} x,\]
and since $\phi\le \lt(\frac{8}{1+16(N-1)}\rt)^2$ implies $5/6 \phi^{1/2}\le 1/7$, we have 
\begin{align*}
 h_N(\phi)&\ge 1-3(N-1)\phi+(\phi^{1/2}-1)(1+\frac{35}{36} \phi^{1/2})\\
 &\ge \frac{1}{36}\phi^{1/2} -3(N-1)\phi.
\end{align*}
In particular, if $\phi\le \frac{1}{108 (N-1)}$, then $h_N\ge 0$. To sum up, we have proven that
\[f'_{N}\ge 0 \qquad \textrm{in } \qquad \lt[0, \min\lt( \lt(\frac{8}{1+16(N-1)}\rt)^2,  \frac{1}{108 (N-1)} \rt)\rt]. \]
Letting $\eta:= \frac{1}{540}$ (which corresponds to $ \frac{1}{108 (N-1)}$ with $N=6$), then for $N\in [3,6]$, $\eta\le  \min\lt( \lt(\frac{8}{1+16(N-1)}\rt)^2,  \frac{1}{108 (N-1)} \rt)$, so that $f'_{N}\ge 0$ in $[0,\eta]$.\\
\smallskip

We finally estimate  $|f'_{N}|$ in $[\eta, N^{-1}]$. Let  
\[g_{N}(\phi):= 1-(N-1) \phi^{3/2}-(1-(N-1)\phi)^{3/2},\]
so that by \eqref{fprime}
\begin{multline}f'_{N}(\phi)= \frac{3(N-1)}{g_{N}(\phi)}\lt( -\phi^2 +(1-(N-1)\phi)^2 \rt.\\ \label{estimphig}
\lt.-\frac{1}{2} (-\phi^{1/2}+(1-(N-1)\phi)^{1/2})\frac{1-(N-1)\phi^{3}-(1-(N-1)\phi)^3}{g_{N}(\phi)}\rt).\end{multline}
In particular, we need to study $\min g_{N}$. Taking the derivative, we obtain
\[g'_{N}(\phi)=\frac{3(N-1)}2(-\phi^{1/2}+(1-(N-1)\phi)^{1/2}).\]
Therefore,  $g'_{N}$ is zero only if $\phi=1/N$ so that  $g'_{N}$ is first positive and then negative and $g_{N}$ attains its minimum on the boundary so that 
\begin{align*}\min_{[\eta,N^{-1}]} g_{N}(\phi)&=\min( g_{N}(\eta), g_{N}(1/N) )\\
 &=\min\lt(1-(N-1)\eta^{3/2}-(1-(N-1)\eta)^{3/2}, 1-N^{-1/2}-2N^{-3/2}\rt)\\
 &\ge 1-5\eta^{3/2}-(1-5\eta)^{3/2}\ge 1,3 \times 10^{-2}.
\end{align*}
Injecting this into \eqref{estimphig}, we obtain
 \begin{align*}
 |f'_{N}(\phi)|&\le \frac{3(N-1)}{1,3 \times 10^{-2}}\lt( 2N^{-2}+ \frac{1}{1,3}N^{-1/2} \times 10^{2}  \rt)\\
 &\le  1,2\times 10^5.
 \end{align*}
We can thus take
\[
 \delta:=10^{-6}< 10^{-2} \sup_{[\eta,N^{-1}]} |f'_N|.
\]

Using a simple Scilab code, we find the values of $\alpha_N$ given in Table \ref{numerictable}.
\begin{table}
\begin{center}
\begin{tabular}{|c|c|c|c|c|}
  \hline
 $N$ & 3 & 4&5&6 \\ 
  \hline
  $\alpha_N\simeq$ & 2& 1.88 & 1.74& 1.64\\
  \hline 
\end{tabular}
\caption{Numerical values of $\alpha_N$}
\label{numerictable}
\end{center}
\end{table}
We see that the values we find are well above the critical values given in Table \ref{critictable}.
\end{proof}
\begin{remark}
 Although we are not able to prove it, the numerics show that for $N\ge 4$, the minimum in $\alpha_N$ is attained for equidistributed masses i.e. $\phi=1/N$. 
\end{remark}
\begin{remark}
 Arguing as in \textit{Step 2}, it could have  been proven that also for the original minimization problem $E(T)$, the optimal masses may take at most two distinct values. 
\end{remark}

 \bibliographystyle{plain}
\bibliography{Gtrans}

\begin{thebibliography}{10}

\bibitem{AGS}
L.~Ambrosio, N.~Gigli, and G.~Savar{\'e}.
\newblock {\em Gradient flows in metric spaces and in the space of probability
  measures}.
\newblock Lectures in Mathematics ETH Z\"urich. Birkh\"auser Verlag, Basel,
  2005.

\bibitem{BeCaMo}
M.~Bernot, V.~Caselles, and J.-M. Morel.
\newblock {\em Optimal transportation networks}, volume 1955 of {\em Lecture
  Notes in Mathematics}.
\newblock Springer-Verlag, Berlin, 2009.
\newblock Models and theory.

\bibitem{BrBuS}
A.~Brancolini, G.~Buttazzo, and F.~Santambrogio.
\newblock Path functionals over {W}asserstein spaces.
\newblock {\em J. Eur. Math. Soc. (JEMS)}, 8(3):415--434, 2006.

\bibitem{BraSol}
A.~Brancolini and S.~Solimini.
\newblock Fractal regularity results on optimal irrigation patterns.
\newblock {\em J. Math. Pures Appl.}, 102:854--890, 2014.

\bibitem{BraWirth}
A.~Brancolini and B.~Wirth.
\newblock Optimal energy scaling for micropatterns in transport networks.
\newblock {\em SIAM J. Math. Anal.}, 49(1):311--359, 2017.

\bibitem{BraSa}
L.~Brasco and F.~Santambrogio.
\newblock An equivalent path functional formulation of branched transportation
  problems.
\newblock {\em Discrete Contin. Dyn. Syst.}, 29(3):845--871, 2011.

\bibitem{ChoKoOtmicro}
R.~Choksi, R.~V. Kohn, and F.~Otto.
\newblock Domain branching in uniaxial ferromagnets: a scaling law for the
  minimum energy.
\newblock {\em Comm. Math. Phys.}, 201:61--79, 1999.

\bibitem{Conti2006}
S.~Conti.
\newblock A lower bound for a variational model for pattern formation in
  shape-memory alloys.
\newblock {\em Cont. Mech. Thermod.}, 17:469--476, 2006.

\bibitem{CoDiZw}
S.~Conti, J.~Diermeier, and B.~Zwicknagl.
\newblock Deformation concentration for martensitic microstructures in the
  limit of low volume fraction.
\newblock {\em Calc. Var. Partial Differential Equations}, 56(1):Art. 16, 24,
  2017.

\bibitem{CoGoOtSe}
S.~Conti, M.~Goldman, F.~Otto, and S.~Serfaty.
\newblock A branched transport limit of the {G}inzburg-{L}andau functional.
\newblock {\em preprint}, 2017.

\bibitem{CoOtSer}
S.~Conti, F.~Otto, and S.~Serfaty.
\newblock Branched microstructures in the {G}inzburg-{L}andau model of type-{I}
  superconductors.
\newblock {\em SIAM J. Math. Anal.}, 48(4):2994--3034, 2016.

\bibitem{KohnMuller94}
R.~V. Kohn and S.~M\"uller.
\newblock Surface energy and microstructure in coherent phase transitions.
\newblock {\em Comm. Pure Appl. Math.}, 47:405--435, 1994.

\bibitem{MSM}
F.~Maddalena, S.~Solimini, and J.-M. Morel.
\newblock A variational model of irrigation patterns.
\newblock {\em Interfaces Free Bound.}, 5(4):391--415, 2003.

\bibitem{MaMa}
A.~Marchese and A.~Massaccesi.
\newblock An optimal irrigation network with infinitely many branching points.
\newblock {\em ESAIM: COCV}, 2014.

\bibitem{ViehOtt}
F.~Otto and T.~Viehmann.
\newblock Domain branching in uniaxial ferromagnets: asymptotic behavior of the
  energy.
\newblock {\em Calc. Var. Partial Differential Equations}, 38(1-2):135--181,
  2010.

\bibitem{PaoSteTep}
E.~Paolini, E.~Stepanov, and Y.~Teplitskaya.
\newblock An example of an infinite steiner tree connecting an uncountable set.
\newblock {\em Adv. Calc. Var.}, 2015.

\bibitem{Sant}
F.~Santambrogio.
\newblock {\em Optimal transport for applied mathematicians}, volume~87 of {\em
  Progress in Nonlinear Differential Equations and their Applications}.
\newblock Birkh\"auser/Springer, Cham, 2015.
\newblock Calculus of variations, PDEs, and modeling.

\bibitem{Viehmanndiss}
T.~Viehmann.
\newblock {\em Uniaxial Ferromagnets}.
\newblock PhD thesis, Universit\"at Bonn, 2009.

\bibitem{villani}
C.~Villani.
\newblock {\em Topics in optimal transportation}, volume~58 of {\em Graduate
  Studies in Mathematics}.
\newblock American Mathematical Society, Providence, RI, 2003.

\bibitem{Xia}
Q.~Xia.
\newblock Interior regularity of optimal transport paths.
\newblock {\em Calc. Var. Partial Differential Equations}, 20(3):283--299,
  2004.

\bibitem{Zwicknagl2014}
B.~Zwicknagl.
\newblock Microstructures in low-hysteresis shape memory alloys: Scaling
  regimes and optimal needle shapes.
\newblock {\em Arch. Ration. Mech. Anal.}, 213:355--421, 2014.

\end{thebibliography}
\end{document}